\documentclass{amsart}
\usepackage{hyperref}
\usepackage{amssymb,amsmath}
\usepackage{enumerate}
\usepackage{color}
\RequirePackage{fix-cm}
\usepackage{graphicx}
\usepackage[cmtip,all]{xy}

\usepackage[draft]{todonotes}
\usepackage[cmtip,all]{xy}

\usepackage{hyperref}
\usepackage[utf8]{inputenc}
\usepackage{amsmath,amsfonts,amssymb,euscript,graphicx,color,tikz}

\newtheorem{thm}{Theorem}[section]
\newtheorem{cor}[thm]{Corollary}
\newtheorem{lem}[thm]{Lemma}
\newtheorem{pro}[thm]{Proposition}

\theoremstyle{definition}
\newtheorem{DEF}[thm]{Definition}
\newtheorem{rem}[thm]{Remark}

\newtheorem{cons}[thm]{Consequence}

\numberwithin{equation}{section}

\def\andd{\quad\hbox{and}\quad}

\def\v{{\mathcal V}}

\def\vd{\dot{\mathcal V}}

\def\vt{\tilde{\mathcal V}}

\def\fm{(\cdot,\cdot)}
\def\a{\alpha}

\def\w{{\mathcal W}}

\def\sub{\subseteq}
\def\rd{\dot{R}}

\def\rt{\tilde{R}}

\def\lam{\lambda}
\def\Lam{\Lambda}

\def\1k{\frac{1}{k}}

\def\la{\langle}
\def\ra{\rangle}
\def\rds{\dot{R}_{sh}}
\def\rdl{\dot{R}_{lg}}

\def\GL{GL}

\def\d{\delta}

\def\b{\beta}

\def\qed{\hfill$\Box$}

\def\sg{\sigma}

\def\bd{\dot{\b}}

\def\hh{{\mathcal H}}

\def\sg{\sigma}

\usepackage{verbatim} %Allows use of \begin{comment} ... \end{comment}

\def\ad{\hbox{ad}}

\def\bbbc{{\mathbb C}}
\def\bbbz{{\mathbb Z}}

\def\bbbr{{\mathbb R}}

\def\bbbn{{\mathbb N}}

\def\ep{\epsilon}

\def\ll{{\mathcal G }}

\def\ll{\mathcal L}

\def\proof{{\noindent\bf Proof. }}

\def\rds{\dot{R}_{sh}}
\def\rdl{\dot{R}_{lg}}
\def\rde{\dot{R}_{ex}}
\def\St{\mathfrak{St}}

\def\G{\mathfrak{G}}
\def\SL{\text{SL}_{2}}

\begin{document}

%%%%%%%%%%%%%%%%%%%%%%%%%%%%%%%%%%%%%%%%%%%%%%%%%%
% Put your own personal macros here, or \input a file of them.
%
%
%%%%%%%%%%%%%%%%%%%%%%%%%%%%%%%%%%%%%%%%%%%%%%%%%%
%%%For Title Page %%%%%%%%%%%%%%%%%%%%%%%%%%%%%%%%
%%%%%%%%%%%%%%%%%%%%%%%%%%%%%%%%%%%%%%%%%%%%%%%%%%

\title{Groups of extended affine Lie type}

\author{Saeid Azam, Amir Farahmand Parsa}
\address
{Department of Mathematics\\ University of Isfahan\\Isfahan, Iran,
P.O.Box: 81745-163 and\\
School of Mathematics, Institute for
Research in Fundamental Sciences (IPM), P.O. Box: 19395-5746.
} \email{azam@ipm.ir}
\address{Institute for
Research in Fundamental Sciences (IPM), P.O. Box: 19395-5746.}\email{a.parsa@ipm.ir}

 \thanks{This research was in part supported by
		a grant from IPM and carried out in
		IPM-Isfahan Branch.}

%%%%%%%%%%%%%%%%%%%%%%%%%%%%%%%%%%%%%%%%%%%%%%%%%%

\begin{abstract}
We construct certain Steinberg groups associated to extended affine Lie algebras
and their root systems.\ Then by the integration methods of Kac and Peterson
for integrable Lie algebras, we associate a group to every tame extended
affine Lie algebra.\ Afterwards, we show that the extended affine Weyl group
of the ground Lie algebra can be recovered as a quotient group of two subgroups
of the group associated to the underlying algebra similar to Kac--Moody groups.
\end{abstract}
\maketitle

\section{Introduction}
\label{intro}
In 1985 K. Saito introduced {\it extended affine root systems} (EARS) in order to
present a suitable geometric space for modeling the universal deformation of a simple elliptic singularity (see \cite{Sai85}).\ A few years later, two mathematical physicists, H\o egh-Krohn and  Toresani, constructed certain Lie algebras whose root systems resembled some of 
the extended affine root systems (see \cite{HoTor90}).\ These algebras were first called \textit{quasi-simple Lie algebras} but, because of their root systems, later they were called \textit{{extended affine Lie algebras}} (EALA).\ These algebras and their root systems are thoroughly studied in \cite{AABGP97}.\

Although there is {a rather rich literature} on EALAs, their associated Weyl groups and EARSs, there is little {about their groups (see \cite{All02,AG01,Az999,Az99, Az00,Kry95,Pian02}).}\ Concerning the construction of groups associated to EALAs and EARSs we can notably mention two works: \cite{Kry95} and \cite{Mor06}.\ For affine Kac--Moody algebras, and in general, all Kac--Moody algebras
there are various methods to associated a group structure to these algebras.\

In \cite{Kac85} Kac gives a recipe for a group construction associated to \textit{integrable} Lie algebras by means of representation of Lie algebras.\ Later on, Tits in \cite{Tits87} associates a group construction by generators and relations {to} Kac--Moody root systems similar to the methods used for finite root systems introduced by Chevalley and Demazure (see \cite{Che55,Che95,De65}).\  

Considering EALAs as a certain generalization of affine Kac--Moody
algebras,  by following \cite{Tits87}, we define certain groups associated to EALAs and EARSs by generators and relations.\ These groups are called \textit{Steinberg}.\ For the simply-laced cases this was done in \cite{Kry95} and we generalize the results therein for arbitrary {EALAs}.\ 

By definition, EALAs are integrable in the sense of \cite{Kac85}, hence we use
Kac and Peterson methods to produce ``integration'' groups for such EALAs.\ Furthermore, we find an epimorphism from Steinberg groups to these ``integration'' groups associated to a fixed EALA.\
This epimorphism together with the generalized presentation by conjugations of the extended affine Weyl groups (see \cite{Az00}) enables us to show that extended affine Weyl groups can be recovered as a quotient group of two subgroups of the ``integration'' groups.\ This generalizes similar results in \cite{Kry95} from simply laced to arbitrary types.

\section{General setup}\label{general-setup}

Let $\bbbc$ be the {field} of complex numbers.\ All vector spaces and algebras are assumed to be over $\bbbc$ unless
otherwise stated.\ By $\la T\ra$, we mean the subgroup generated by a subset $T$ of the ground group or vector space.\ For a vector space $\v$ over $\bbbr$ equipped with a symmetric form $\fm$, we set $\v^0$ to be the 
radical of the form, also  for $R\sub\v$, we set $R^0=R\cap\v^0$ and $R^\times=R\setminus R^0$.
For $\a\in \v$ with $(\a,\a)\neq 0$, we set $\a^\vee:=2\a/(\a,\a)$, {and define the reflection $w_\a\in\GL(\v)$ by $w_\a(\b)=\b-(\beta,\alpha^\vee)\a$. Sometimes we also use the notation $\la\a,\b\ra:=(\a,\b^\vee).$}  

Let $(\ll,\fm,\hh)$ be a tame irreducible extended affine Lie algebra (EALA). This means that $\ll$ is a Lie algebra, $\hh$ is a non-trivial
subalgebra of $\ll$ and $\fm$ is a symmetric bilinear form on $\ll$ satisfying (A1)-(A5) below.

(A1) $\fm$ is invariant and non-degenerate,

(A2) $\hh$ is a finite-dimensional Cartan subalgebra of $\ll$.

Axiom (A2) means that $\ll=\sum_{\a\in\hh^\star}\ll_\a$ with $\ll_\a=\{x\in\ll\mid [h,x]=\a(h)x\hbox{ for all }h\in\hh\},$
and $\ll_0=\hh$. Let $R$ be the set of roots of $\ll$, namely $R=\{\a\in\hh^\star\mid\ll_\a\neq\{0\}\}$.
It follows from (A1)-(A2) that the form $\fm$ restricted to $\hh$ is non-degenerate and so it can be transferred
to $\hh^\star$ by $(\a,\b):=(t_\a,t_\b)$ where $t_\a\in\hh$
is the unique element satisfying $\a(h)=(h,t_\a)$, $h\in\hh$.
%Let us denote by ${\mathcal R}$ the radical of the form on $\hh^\star$.
Let $R^0=\{\a\in R\mid (\a,\a)=0\}$ and $R^\times=R\setminus R^0$.
Then $R=R^0\uplus R^\times$ is regarded as the decomposition of roots into
{\it isotropic} and {\it non-isotropic} roots, respectively. Let $\v:=\hbox{Span}_\bbbr R$
and $\v^0:=\hbox{Span}_{\bbbr}R^0$.

(A3) For $\a\in R^\times$ and $x\in\ll_\a$, $\ad(x)$ acts locally nilpotently on $\ll$.

(A4) The $\bbbz$-span of $R$ in $\hh^\star$ is a free abelian group of rank $\dim\v$.

(A5) (a) $R^\times$ is indecomposable,

(b) $R^0$ is {\it non-isolated}, meaning that $R^0=(R^\times-R^\times)\cap\v^0.$

The {\it core} of an EALA $\ll$ is by definition, the subalgebra $\ll_c$ of $\ll$ generated by non-isotropic root spaces.\
It follows that $\ll_c$ is a perfect ideal of $\ll$.
When $\ll_c^\perp:=\{x\in\ll\mid (x,\ll_c)=\{0\}\}$ is contained in the core,  $\ll$ is called {\it tame}.

Let $(\ll,\fm,\hh)$ be an EALA.\
Let $\w_\ll$ be the Weyl group of $\ll$; the subgroup of $\GL(\hh^\star)$ generated by reflections $w_\a$, $\a\in R^\times$.\ By \cite[Theorem I.2.16]{AABGP97}, the root system $R$ of $\ll$  is an irreducible extended affine root system in the following sense.

\begin{DEF}\label{def2}\rm{
		Let $\v$ be a finite-dimensional real vector space equipped with a non-trivial positive semidefinite symmetric bilinear form $\fm$, and $R$ a subset of $\v$. Triple $(\v,\fm,R)$, or $R$ if there is no confusion, is called an {\it extended affine root system} (EARS)
		{if the following axioms hold:}}
	{\begin{itemize}
			\item[(R1)] $\la R\ra$ is a full lattice in $\v$,
			\item[(R2)] $(\b,\a^\vee)\in\bbbz$, $\a,\b\in R^\times$,
			\item[(R3)] $ w_\a(\b)\in R$ for $\a\in R^\times$, $\b\in R$,			
			\item[(R4)] $R^0= \v^0\cap(R^\times - R^\times)$,
			\item [(R5)] $\a\in R^\times\Rightarrow 2\a\not\in R$.
		\end{itemize}
		%$R$ is called {\it reduced} if $\a\in R^\times$ implies $2\a\not\in R^\times$, and 
We say that $R$ is {\it irreducible} or
		{\it connected} if  $R^\times$ cannot be written as the union of two of its non-empty orthogonal subsets.}
\end{DEF}		
\begin{rem}\label{rem100}{\rm
		One checks that the definition of an irreducible EARS given above coincides with the definition of an extended affine root system  given 
		in \cite[Definition II.2.1]{AABGP97}.}
\end{rem}

Let $(\v,\fm,R)$ be an {irreducible} EARS.\ It follows that the canonical image $\bar R$ of $R$ in $\bar\v:=\v/\v^0$ is a finite root system
in $\bar\v$. The {\it type} and the {\it rank} of $R$ is defined to be the type and the rank of $\bar R$, {and the dimension of $\v^0$ is called the {\it nullity} of $R$.}
{$R$ is called {\it reduced} if $\bar R$ is reduced.}
Taking an appropriate pre-image $\dot R$ of $\bar R$ in $\v$, under the canonical map $\v\longrightarrow\bar\v$,
one can see that $\rd$ is a finite root system in $\vd:=\hbox{Span}_\bbbr\dot R$ isomorphic to $\bar R$.\ Moreover, one obtains a
description of $R$ in the form
\begin{equation}\label{eq99}
R=R(\rd,S,L,E):=(S+S)\cup (\rds+S)\cup (\rdl+L)\cup (\rde+E),
\end{equation}
where $\rds$, $\rdl$ and $\rde$ are the sets of short, long and extra long roots of $\rd$, respectively, and $S$, $L$ and $E$ are
certain subsets of $R^0$, called {\it (translated) semilattices}, which interact in a prescribed way (see \cite[Chapter II]{AABGP97} for details).\ {If $R$ is reduced, then in (\ref{eq99}), $\rd_{ex}$ and $E$ are interpreted as empty sets, so
	\begin{equation}\label{eq99-1}
	R=R(\rd,S,L)=(S+S)\cup (\rds+S)\cup (\rdl+L).
	\end{equation}
	for some  reduced finite root system $\rd$}.\
For $\a\in R^\times$, we denote by $\dot\a$ the unique element in $\rd$ associated to $\a$ via the description (\ref{eq99}), namely $\a=\dot\a+\d_\a$ for $\dot\a\in\vd$ and $\d_\a$ in $S$, $L$ or $E$.

We have $\v=\vd\oplus\v^0.$ Set $\vt:=\v\oplus (\v^0)^\star$. We extend the form on $\v$ to a non-degenerate form on $\vt$ by the natural dual pairing
of $\v^0$ and $(\v^0)^\star$, namely $(\lam,\sg):=\lam(\sg)$ for $\lam\in (\v^0)^\star$ and $\sg\in\v$.
Let $\w$ be the subgroup of $\GL(\vt)$ generated by reflections $w_\a$, $\a\in R^\times$. {If $R$ is the root system of an 
	extended affine Lie algebra,} then
$\w\cong\w_\ll$. We refer to $\w$ as the {(extended affine)} Weyl group of $R$.
{Extended affine Weyl groups are not in general Coxeter groups {\cite[Theorem 1.1]{Hof07}}.\
	%	Note that $\w$ does not have the deletion condition, hence it is actually not a Weyl group.\ 
	However, they enjoy a related presentation called ``generalized presentation by conjugation"  which will be used in the sequel. We recall briefly this presentation here.}
Assume that $R$ is reduced and that
$\theta_s$ and $\theta_\ell$ are the highest short and highest long roots of $\rd$. One knows that $S$ contains a $\bbbz$-basis $\{\sg_1,\ldots,\sg_n\}$ of the lattice $\Lam:=\la R^0\ra=\la S\ra$ (see \cite[Proposition II.1.11]{AABGP97}), {and that if $\rdl\neq\emptyset$ then $|\Lam/\la L\ra|=k^t$, where $k\in \{2,3\}$ (see \cite[Chapter II.4]{AABGP97}). The integer $t$ is called the {\it twist number} of $R$.} 

For $\a\in R^\times$ and $\sg={\sum^{n}_{i=1}m_i\sg_i\in R^0}$ with
$\a+\sg\in R$, we set
\begin{equation}\label{b4}
c_{(\a,\sg)}:=
(w_{\a+\sg}w_{\a}) (w_{\a}w_{\a+\sg_1})^{m_1}(w_{\a}w_{\a+\sg_2})^{m_2}
\dots (w_{\a}{w_{\a+\sg_n})^{m_n}.}
\end{equation}
%By \cite[\S2]{Az00}, $c_{(\a,\sg)}$ is an element of the Weyl group.\
Consider a pair
\begin{eqnarray*}
	(\a_p,\eta_p = \sum^n_{i=1} m_{ip}\sg_i) \in (\{\theta_s\} \times
	\sum^n_{i=1} {\bbbz}
	\sg_i) \bigcup (\{\theta_\ell\} \times \sum^n_{i=t+1} {\bbbz}\sg_i),
\end{eqnarray*}
and let $\ep_p\in \{\pm 1\}$. The triple $\{(\ep_p,\a_p,\eta_p)\}^m_{p=1}$ is called a
{\it reduced collection} if
$$
\sum^m_{p=1} k(\a_p) \ep_p m_{ip} m_{jp} =  0 \quad \hbox{for all~} 1\leq
i<j\leq n,
$$
where $k(\theta_\ell)=1$ {for all types,} and $k(\theta_s)=3$ for type $G_2$ and $k(\theta_s)=2$ for the remaining types.
Here we have used the convention that if the type of $R$ is simply laced, then $\theta:=\theta_s=\theta_\ell$ and
$k(\theta)=1$.\ We note from  \cite[\S2]{Az00} that if $\{(\ep_p,\a_p,\eta_p)\}^m_{p=1}$ is a
reduced collection then
$c_{(\a_p,\eta_p)}$ is an element of $\w$. We record here the following fact which will be used in the sequel.
\begin{thm}\label{azam2000}\cite[Theorem 3.7]{Az00}
	The Weyl group $\w$ is isomorphic to the group $\hat\w$
	defined by generators $\hat{r}_\a,$ $\a\in R^\times$ and relations:
	\begin{itemize}
		\item[(i)] $\hat{r}^2=1;$ $\a\in R^\times$,
		\item[(ii)] $\hat{r}_\a\hat{r}_\b\hat{r}_\a=\hat{r}_{r_\a(\beta)};$ $\a,\b\in R^\times,$
		\item[(iii)] $\prod_{p=1}^m\hat{c}_{(\a_p,\eta_p)}= 1$ for any reduced collection $\{(\ep_p, \a_p, \eta_p)\}_{p=1}^m$,
	\end{itemize}
	where $\hat{c}_{(\a_p,\eta_p)}$ is the element in $\hat\w$ corresponding to $c_{(\a_p,\eta_p)}$, {under the assignment
		$w_\a\mapsto\hat{w}_\a$}.
\end{thm}
The presentation given in Theorem \ref{azam2000} is called the {\it generalized presentation by conjugation} for $\w$. If $\w$ is  
isomorphic to the presented group determined by (i)-(ii) of Theorem \ref{azam2000}, 
then $\w$ is said to have the {\it presentation by conjugation}.\ Given an EARS $R$, there is a computational procedure to decide
whether its Weyl group has presentation by conjugation or not, see {\cite{AS11}}. In particular, it is known that all EARS of nullity $\leq 2$ have presentation by conjugation.

\section{Nilpotent pairs}\label{nilpotent-pairs}	
For the rest of this work, we assume that $(\ll,\fm,\hh)$ is a tame extended affine Lie algebra and $(\v,\fm,R)$ is an irreducible reduced extended affine root system.\ Whenever, we work with $\ll$, we assume $R$ is its root system.\ We proceed with the same notation as in Section \ref{general-setup}.
%% in particular we have $\v=\vd\oplus\v^0$ and
%$$
%	R=R(\rd,S,L):=(S+S)\cup (\rds+S)\cup (\rdl+L).
%$$	
\begin{DEF}\label{newdef2}{\rm
		We say a subset $T$ of $R$ is a {\it subsystem} of $R$ if $T$ is an EARS in $\v_T:=\hbox{span}_\bbbr T$ with respect to the form induced 
		from $\v$ on $\v_T$. In particular, we say $T$ is a finite or an affine subsystem of $R$, if $T$ is a finite or an affine root system on its own.} 
\end{DEF}

\begin{lem}\label{lem99} {Let $T$ be an irreducible finite subsystem of $R$.}\ Then the subalgebra
	$\ll_T$ of $\ll$ generated by $\ll_{\pm\a}$, $\a\in T\setminus\{0\}$, is a finite-dimensional simple Lie algebra of type $T$.\ Moreover,
	$$\ll_T=\sum_{\a\in T\setminus\{0\}}\ll_{\a}\oplus\sum_{\a\in T}[\ll_\a,\ll_{-\a}].$$
\end{lem}
\proof It follows from \cite[Theorem 1.29]{AABGP97} that $$M_T:=\sum_{\a\in S\setminus\{0\}}\ll_\a\oplus\sum_{\a\in T}[\ll_\a,\ll_{-\a}]$$ is a finite-dimensional simple Lie algebra of type $T$. Since the roots appearing in $\ll_T$ belong to the root strings of roots in $T$, and $T$ as a finite root system is closed under root strings, we conclude that $\ll_T\sub M_T$. Clearly $M_T\sub\ll_T$.\qed

Let $T$ be a subsystem of $R$ and $\lam\in\Lam$. We set
\begin{equation}\label{affine-100}
T_{\lam}:=(T^\times+\bbbz\lam)\cap R\andd R_{T,\lam}:=T_\lam\cup \big( (T_\lam-T_\lam)\cap R^0\big).
\end{equation}
\begin{pro}\label{affine-999} Let $T$ be a finite irreducible subsystem of $R$ and $0\neq\lam\in\Lam$.
	Then $R_{T,\lam}$ is an affine irreducible subsystem of $R$.
\end{pro}
\proof Let $\v_{\lam}:=\hbox{span}_{\bbbr}R_{T,\lam}$ and $\v^0_\lam:=\v_\lam\cap\v^0.$ Since
$2\Lam\sub S\sub R^0$ we have,  for  $\gamma\in R^\times,$
$$\gamma+2\Lam\sub\rds+S+2\Lam\sub\rds+S\sub R^\times\hbox{ if }\gamma\hbox{ is short,}$$ and
$$\gamma+2\Lam\sub\rdl+L+2\Lam\sub\rdl+L\sub R^\times\hbox{ if }\gamma\hbox{ is long.}$$ Thus
$T^\times+2\bbbz\lam\sub T_\lam$ and so $2\bbbz\lam\sub (T_\lam-T_\lam)\cap R^0\sub\v_\lam^0$.
It follows from this and the fact that $T\cap\v^0=\{0\}$ that $\v^0_\lam=\bbbr\lam$, namely the radical of the form restricted to $R_{T,\lam}$ is one-dimensional.\
Using this and \cite[Theorem 2.31]{ABGP97}, {it suffices to} show that
$R_{T,\lam}$ is {an irreducible} EARS.\ The irreducibility of $R_{T,\lam}$ follows immediately from the irreducibility of $T$, so it is enough to show that {(R1)-(R5)} hold.\
Axioms (R1) and (R2) clearly hold for $R_{T,\lam}$.\ Axiom (R4) holds as $R_{T,\lam}^0=
(T_\lam-T_\lam)\cap R^0$.\ {Since by assumption $R$ is reduced, (A5) holds.} Next,
if $\gamma+n\lam,\gamma'+n'\lam\in R^\times_{T,\lam}$, where $\gamma,\gamma'\in T^\times$,
then
$$w_{\gamma+n\lam}(\gamma'+n'\lam)=w_{\gamma}(\gamma')+n'\lam+(\gamma',\gamma^\vee)n\lam\in R\cap
(T^\times+\bbbz\lam)= R^\times_{T,\lam}.$$
Thus (R3) holds and  the proof is complete.\qed

Let $T$ be an irreducible finite subsystem of $R$ and $0\neq\d\in R^0$.\ By Proposition \ref{affine-999}, $R_{T,\d}$ is an irreducible affine subsystem of $R$.\ By a standard argument (see \cite[Chapter I]{AABGP97}), one can pick an element $\gamma\in\hh^\star$ such that
$(\gamma,\a)=(\gamma,\gamma)=0$ for $\a\in T$ and $(\gamma,\d)\neq 0$. Let $d=t_\gamma$,  the unique element in $\hh$ which represents
$\gamma$ via the form.\ Set
\begin{equation}\label{new12}
\ll_{T,\delta}:=\sum_{0\neq\a\in R_{T,\delta}}\ll_\a+\sum_{\a\in R_{T,\delta}}[\ll_\a,\ll_{-\a}]+\bbbc d.
\end{equation}	

\begin{lem}\label{lemafsubLi}
	Let $T$ be an irreducible finite subsystem of $R$ and $\d\in\Lam$. 
	%$R_{\a,\b,\delta}$ be as in Proposition~\ref{affine-99.
	Then $\ll_{T,\delta}$
	%:=\sum_{\gamma\in R_{\a,\b,\delta}\setminus\{0\}}\ll_\gamma+\sum_{\gamma\in R_{\a,\b,\delta}}[\ll_\gamma,\ll_{-\gamma}],$$
is an extended affine Lie algebra of nullity $1$ with root system $R_{T,\delta}.$ Moreover the subalgebra 
$\ll_{T,\delta}^{a}=\la\ll_\a\mid\a\in R^\times_{T,\delta}\ra\oplus\bbbc d$ of $\ll_{T,\delta}$ constitutes an affine Lie algebra with root system $R_{T,\delta}$.  
\end{lem}

\proof 
By \cite{AABGP97}, we have 
$$\ll_{T,\d}=\hh_{T,\d}\oplus\sum_{0\neq\a\in R_{T,\delta}}\ll_\a
%\oplus\sum_{\sg\in R^0_{T,\delta}}\ll_\sg\cap\ll_{T,\delta}
$$ where
$\hh_{T,\d}=\sum_{\a\in T}\bbbc t_\a\oplus\bbbc t_\d\oplus\bbbc d.$ From the way  $d$ is
chosen, the form is non-degenerate (and clearly invariant) on both $\ll_{T,\d}$ and $\hh_{T,\d}$.\ In particular, we have
$\hh=\hh_{T,\d}\oplus\hh_{T,\d}^{\perp}$.\ From this it follows that $\ll_{T,\d}$ has a weight space decomposition
with respect to $\hh_{T,\d}$ with the set of roots $R_{T,\d}$ with  $({\ll_{T,\d}})_0=\hh_{T,\d}$.\ Thus the axioms
(A1) and (A2) of an EALA hold for $\ll_{T,\d}$. The remaining axioms of an EALA holds trivially for $\ll_{T,\d}$.
Since $R_{T,d}$ is an affine root system, $\ll_{T,\d}$ has nullity $1$.

Next, we consider the subalgebra $\ll_{T,\delta}^a$. If $\sg\in R^0$ and $\a,\a+\sg\in R^\times$, then from the invariance of the form and one-dimensionality of non-isotropic root spaces it follows that
$([\ll_{\a+\sg},\ll_{-\a}],[\ll_{-\a-\sg},\ll_{\a}])\not=0$. From this, we conclude that the form restricted to $\ll_{T,\delta}^a$ is also non-degenerate.
Now an argument similar to the first statement shows that $\ll^a_{T,\delta}$ is an extended affine Lie algebra of nullity $1$, with root system
$R_{T,\delta}$. Therefore by \cite[Theorem 2.31]{ABGP97}, it remains to show that  $\ll^a_{T,\d}$ is tame.\ For this,
let  $\ll^a_{T,\d,c}$ denote the core of $\ll^a_{T,\delta}$ and $x\in\ll^a_{T\,\delta}$ with
$(x,\ll^a_{T,\delta,c})=\{0\}$.   We must show $x\in\ll^a_{T,\delta,c}$. Since the form on $\ll$ is $\hh^\star$-graded, namely $(\ll_\a,\ll_\b)=\{0\}$ unless $\a+\b=0$, we get
$(x,\ll_\b)=0$ for all $\b\in R^\times,$ implying that  $x\in\ll^\perp_c\sub\ll_c$.
So $x\in\ll_c\cap\ll_{T,\delta}^a$, forcing $x\in\ll^a_{T,\delta,c}$ as required.
\qed

\begin{lem}\label{lemr2srs}
	Let $R$ be a reduced EARS of type $X$.\ Let $\alpha,\beta\in R$ be a pair of non-isotropic roots in $R$ such that $\alpha+\beta$ is also a non-isotropic
	root in $R$.\ Then the set $R_{\alpha,\beta}:=R\cap (\bbbr\a\oplus\bbbr\b)$ is an irreducible reduced finite subsystem of $R$.
\end{lem}
\begin{proof}
	We prove the lemma through proving four claims as follows:\\
	\textbf{Claim 1}: $\bar{\alpha},\bar{\beta}$ are non-proportional in $\bar{\v}$.
	If not, then as $R$ is reduced, the only possibilities are  $\bar{\beta}=\pm\bar{\alpha}$.
	But then in this case, if $\bar{\beta}=-\bar{\alpha}$ then $\alpha+\beta$ has to be
	isotropic which is impossible by the assumption, and, if  $\bar{\beta}=\bar{\alpha}$ then $2\bar\a\in\bar{R}$ which is again impossible.\ Hence Claim 1 follows.\\
	\textbf{Claim 2}: $R_{\alpha,\beta}$ is finite.
	We know that $\bar{R}$ is finite. Therefore, if $R_{\alpha,\beta}$ is not finite then there exist $\dot{\gamma}\in\dot{R}$ and an infinite sequence $\delta_{n}$ of
	isotropic roots in $R$ such that $\dot{\gamma}+\delta_{n}\in R_{\alpha,\beta}$ for every $n\in\mathbb{N}$.\ Then for each $n$, there exist $k_{n},k'_{n}\in\bbbr$
	such that
	\begin{equation}\label{eq:lemr2srs00}
	\dot{\gamma}+\delta_{n}=k_{n}\alpha+k'_{n}\beta.
	\end{equation}
	Now we have $\alpha=\dot{\alpha}+\delta$ and $\beta=\dot{\beta}+\delta'$
	where $\dot\a,\bd\in\rd$ are non-proportional and $\delta,\delta'$ are isotropic.
	By (\ref{eq:lemr2srs00}) we have
	\[
	\dot{\gamma}+\delta_{n}=k_{n}\dot{\alpha}+k'_{n}\dot{\beta}+k_{n}\delta+k_{n}'\delta',\quad\hbox{ for all } n\in\mathbb{N}.
	\]
	Since by Claim 1, $\dot{\alpha},\dot{\beta}$ are non-proportional, this gives
	$k_{n}=k_{m}$ and $k_{n}'=k_{m}'$ for all $n,m\in\mathbb{N}$.\ Thus $\delta_{n}=\delta_m$ for all $n,m\in\mathbb{N}$ which is a contradiction.\\
	\textbf{Claim 3}: $R_{\alpha,\beta}$ is irreducible.
	It is enough to show that $\bar{R}_{\alpha,\beta}$ is irreducible.\ But
	this is immediate as  $\bar{\alpha},\bar{\beta}$ are non-proportional and
	$\bar{\alpha}+\bar{\beta}$ is a root.\\
	\textbf{Claim 4}: $R_{\alpha,\beta}$ is an irreducible finite root system
	in $\v_{\alpha,\beta}:=\bbbr\a\oplus\bbbr\b$. By Claims 2,3, $R$ is finite and irreducible. Moreover, the form restricted to $\v_{\a,\b}$ is positive.\ Now since the subgroup $\w_{\alpha,\beta}$ of the Weyl group of $R$ generated by
	the reflections associated to $\alpha,\beta$ clearly preserves $R_{\alpha,\beta}$, and $(\gamma,\eta^\vee)\in\bbbz$ for all
	$\gamma,\eta\in R_{\a,\b}\setminus\{0\}$, we get from Remark \ref{rem100} that the claim holds.\end{proof}
\begin{cons}\label{Consr2rst}
	In the situation of Lemma~\ref{lemr2srs}, we have $$\text{Type}(R_{\alpha,\beta})=\text{Type}(\dot{R}_{\dot{\alpha},\dot{\beta}}).$$
\end{cons}
\begin{proof}
	By definition, $R_{\a,\b}$  has the same type of
	$\bar{R}_{\a,\b}=\bar{R}\cap (\bbbz\bar{\a}\oplus\bbbz\bar\b)\cong \rd_{\dot\a,\dot\b}.$
\end{proof}

\begin{DEF}\label{defrevise}{\rm
		A pair $\{\a,\b\}$ of non-isotropic roots in $R$ satisfying $\a+\b\in R^\times$  is called
		a {\it nilpotent} pair.}\ {According to Lemma~\ref{lemr2srs}, in this case $R_{\a,\b}$ is an irreducible finite root system.}
\end{DEF}
Let $\{\a\,\b\}$ be a nilpotent pair.\ By Lemmas \ref{lemr2srs} and \ref{lem99}, $\ll_{R_{\a,\b}}$ is a finite-dimensional simple Lie algebra of rank $2$. Let $\ll_{\a,\b}$ be the subalgebra of $\ll$ generated by
$\ll_{\pm\a},\ll_{\pm\b}$.
It follows
from \cite[Lemma 1.3]{AG01}\ {that there exists a base $\{\gamma_1,\gamma_2\}$ of $R_{\a,\b}$ such that
	$\ll_{\gamma_i}\sub\ll_{\a,\b}$, {$i=1,2$.}}\ It now follows again from
\cite[Lemma 1.3]{AG01} that $\ll_{\a,\b}$ contains
all root spaces $\ll_{\gamma}$, $\gamma\in R_{\a,\b}\setminus\{0\}$ and
so $\ll_{\a,\b}=\ll_{R_{\a,\b}}$.\
We denote the nilpotent part of this Lie algebra obtained from
positive roots by $\mathfrak{l}_{\alpha,\beta}$.\ Note that the above argument shows that for each nilpotent pair
$\{\a,\b\}$, we have $\ll_{\a,\b}=\ll_{\a',\b'}$ where $\{\a',\b'\}$ is a base of $R_{\a',\b'}=R_{\a,\b}$.

\begin{pro}\label{affine-99} Let $\{\a,\b\}$ be a nilpotent pair in $R$ and $\d=\d_\a+\d_\b$.	
	Then $\d\in R^0$, moreover if $\d\neq0$, then
	$R_{\a,\b,\d}:=R_{R_{\a,\b},\d}=R(R_{\a,\b}, S_{\a,\b},L_{\a,\b})$
	where $S_{\alpha,\beta}:=\mathbb{Z}\delta$ and
	$$L_{\alpha,\beta}=\left\{\begin{array}{ll}
	\mathbb{Z}\delta&\hbox{if }R_{\alpha,\beta}\hbox{ is simply-laced,}\\
	2\mathbb{Z}\delta&\hbox{if }R_{\alpha,\beta}\hbox{ is of type }B_2=C_2,\\
	3\mathbb{Z}\delta&\hbox{if }R_{\alpha,\beta}\hbox{ is of type }G_2,
	\end{array}\right.
	$$
	(see (\ref{eq99}) for the notation $R(R_{\a,\b},S_{\a,\b},L_{\a,\b})$).
\end{pro}
\proof With the notation of (\ref{eq99}) let $R=R(\dot R,S,L)$. Put $S':=S$ if $\a$ is short in $R$, and put $S':=L$ if $\a$ is long in $R$.
Since $\a+\b\in R$, we have $\d\in R^0$. If $\d\neq 0$, we have from Proposition \ref{affine-999} that 
$R_{\a,\b,\d}$ is an irreducible affine root system. To complete the proof,
we proceed in a few steps according to the type of $R_{\a,\b}$. Suppose first that $R_{\a,\b}$ is simply laced.
As $\a,\b,\a+\b$ are in $R_{\a,\b}$, we have
$R_{\a,\b}=\{\pm\a,\pm\b,\pm(\a+\b)\}$.
Now for any
$n\in\bbbz$, we have
$$\pm\a+n\d
%=\pm\dot\a+\d_\a+n(\d_a+\d_\b)
=\pm\dot\a+(n+1)\d_\a+n\d_\b\in
\pm\dot\a+S'+2S'\sub\pm\dot\a+S'\sub R.$$
This shows that $\pm\a+\bbbz\d=\pm\a+S_{\a,\b}\sub R.$ Similarly, we get $\pm\b+S_{\a,\b}\sub R$. Also we have
$\pm(\a+\b)+S_{\a,\b}=\pm(\dot\a+\dot\b)+\bbbz\d\sub\pm(\dot\a+\dot\b)+S'\sub R$, as required.

Next, suppose that $R_{\a,\b}$ is not simply laced. As discussed in the paragraph preceding the lemma, we may assume that
$\{\a,\b\}$ is a base of $R_{\a,\b}$.  Without loose of generality, assume that $\a$ is short and $\b$ is long.\
If we are in type $B_2$, then $\pm\{\a,\a+\b\}$ and
$\pm\{\b,2\a+\b\}$ are the sets of short and long roots of $R_{\a,\b}$, respectively.\ Then we have
$$
\begin{array}{l}
\pm\a+S_{\a,\b}=\pm(\dot\a+\d_\a)+\bbbz(\d_\a+\d_\b)\sub\pm\dot\a+S+2S\sub\pm\dot\a+S\sub R,\\
\pm(\a+\b)+S_{\a,\b}
=\pm(\dot\a+\dot\b+\d)+\bbbz\d\sub\pm(\dot\a+\dot\b)+S\sub R,\\
\pm\b+L_{\a,\b}=\pm(\dot\b+\d_\b)+2\bbbz\d
\sub\pm\dot\b+L+2L\sub\pm\dot\b+L\sub R,\\
\pm(2\a+\b)+L_{\a,\b}=\pm(2\dot\a+\dot\b+2\d_\a+\d_\b)+2\bbbz\d\\
\hspace{4cm}\sub\pm(2\dot\a+\dot\b)+2S+L+2S
\sub\pm(2\dot\a+\dot\b)+L\sub R.
\end{array}
$$
Finally, we consider the case in which $R_{\a,\b}$ is of type $G_2$. In this case
$\pm\{\a,\a+\b,\b+2\a\}$ and $\pm\{\b+3\a,\b+3\a,2\b+3\a\}$ are the sets of short and long roots of $R_{\a,\b},$
respectively. Now using an argument similar to the case $B_2$, using the facts that $L+3S\sub L$, $S+L\sub S$ and $S+2S\sub S$, we are done.\qed

\begin{lem}\label{lemisoweylrs}
	Notations as in Lemma~\ref{lemr2srs} and Proposition~\ref{affine-99}.\ Let $\w$ be the Weyl group associated to a reduced {EARS $R$}.\ Then for any
	nilpotent pair of roots $\{\a,\b\}$ and any $w\in\w$ we have
	\begin{equation}\label{eq:lemisoweylrs00}
	R_{\a,\b}\cong wR_{\a,\b}=R_{w\a,w\b},
	\end{equation}
	and
	\begin{equation}\label{eq:lemisoweylrs01}
	R_{\a,\b,\delta}\cong wR_{\a,\b,\delta}=R_{w\a,w\b,\delta}.
	\end{equation}
\end{lem}

\proof It follows from the fact that the form is invariant under the action
of the Weyl group, in particular $\w$ preserves root lengths.
\qed

\section{Groups of Extended Affine Lie Type}\label{secGAEI}
\subsection{Groups Associated to {EARSs}}\label{subsecGAERI}
We proceed with the same notations and assumptions as in the previous two sections.\ In particular, we assume $(\ll,\hh,\fm)$ is { a tame EALA with root system $R$}, $R$ is reduced of type $\dot{R}\subset R$ and $\Lambda\sub\hh^\star$ is the lattice generated by isotropic roots.

Let $\bbbc_{\sigma}$ be the associative $\bbbc$-algebra of the crossed product of $\Lambda$ and $\bbbc$ with respect to {a $2$-cocycle}  $\sigma$ (see \cite{Pass89}).\ {Note that the definition of $\bbbc_\sg$ depends on the nullity of $\ll$  (or equivalently the nullity of $R$), and the chosen $2$-cocycle $\sg$.\ We call $\bbbc_{\sigma}$ the \textit{coordinate ring} associated to the $\ll$ (or $R$).\ {The algebra $\bbbc_{\sigma}$ is known in the literature with different names such as $\Lam$-torus, twisted group algebra, or quantum torus  (see \cite[\S 1.4]{AYY15}).}}

Let us recall that as a vector space $\bbbc_{\sigma}$ has a basis $\{c_\lam\mid\lam\in \Lam\}$, and  that the multiplication on $\bbbc_{\sigma}$ is determined by 
$$c_\lam c_\tau = \sg(\lam,\tau)c_{\lam+\tau},\quad (\lam,\tau\in\Lam).$$
{When $R$ is not of type $A_n$ (for some $n\ge2$), we always assume that $\sg$ is commutative.
Note that the above condition on $\sigma$ is not a major restriction on the simply laced cases by \cite[Theorem 40.22(v\&iv)]{Tits02} when we are mainly interested in group structures.}
 
For a pair of roots $\alpha,\beta\in R$ define $R_{\alpha,\beta}^{+}:=(\mathbb{N}\alpha+\mathbb{N}\beta)\cap R$.\ When $R$
is a finite root system, there exists an ordering, {denoted by $<$,} on the set
of positive roots $R^{+}$ such that if $\alpha<\beta$ then $\beta-\alpha\in R^{+}$.\
On such sets of positive roots we always consider this ordering and call it the {\it canonical order}.  
In what follows, the term $(x,y):=xyx^{-1}y^{-1}$ denotes the group
commutator for two elements $x$ and $y$ of a group.
\begin{DEF}\label{DefStfinite}{\rm
		Let $\dot{R}$ be a finite reduced root system.\ Let $A$ be an associative
		unital ring and $U(A)$ the group of units of $A$.\ The group generated by $\hat{x}_{\dot{\alpha}}(a)$ for all $\dot{\alpha}\in\dot{R}^\times$ and $a\in A$ subject to the following relations
		is called the {\it Steinberg group} {of type $\dot{R}$ over $A$} and is denoted by $\St_{\dot{R}}(A)$:
		\begin{description}
			\item[StF1] $\hat{x}_{\dot{\alpha}}(a)$ is additive in $a$.
			\item[StF2] {If $\text{rank}(\dot{R})\geq2$ consider: for a pair of roots $\dot{\alpha},\dot{\beta}\in\dot{R}$, if $\dot{\alpha}+\dot{\beta}\in \dot{R}^\times$, then
				\begin{equation}\label{eq:DefStfinite22}
				(\hat{x}_{\dot{\alpha}}(a),\hat{x}_{\dot{\beta}}(b))=\prod_{i\dot{\alpha}+j\dot{\beta}\in \dot{R}_{\dot{\alpha},\dot{\beta}}^{+}}\hat{x}_{i\dot{\alpha}+j\dot{\beta}}(c_{ij}a^{i}b^{j}),
				\end{equation}
				with the canonical order on the finite set $\dot{R}_{\alpha,\beta}^{+}$, where $c_{ij}$'s in (\ref{eq:DefStfinite22}) are chosen as in \cite[Lemma 15]{St67} corresponding to $i\dot{\alpha}+j\dot{\beta}\in \dot{R}$.}
			\item[StF$\text{2}'$] If $\text{rank}(\dot{R})=1$, then 
			\begin{equation}\label{eq:finiteLWA}
			\hat{n}_{\dot{\alpha}}(a)\hat{x}_{\dot{\alpha}}(b)\hat{n}_{\dot{\alpha}}(a)^{-1}=\hat{x}_{-\dot{\alpha}}(-a^{-2}b),
			\end{equation}
			for $a\in U(A)$ and $b\in A$, where 
			\begin{equation}\label{eq:finiteWF}
			\hat{n}_{\dot{\alpha}}(a):=\hat{x}_{\dot{\alpha}}(a)\hat{x}_{-\dot{\alpha}}(-a^{-1})\hat{x}_{\dot{\alpha}}(a).
			\end{equation}
	\end{description}}
\end{DEF}

{ 
	Let $\hat{N}$ and $\hat{T}$ denote the subgroups of $\St_{\dot{R}}(A)$ generated
	by  $\hat{n}_{\dot{\alpha}}(a)$ and  $\hat{h}_{\dot{\alpha}}(a):=\hat{n}_{\dot{\alpha}}(a)\hat{n}_{\dot{\alpha}}(1)^{-1}$ for $a\in U(A)$, $\dot{\alpha}\in\dot{R}^\times$.\ I.e,
	\begin{equation}\label{eq:MPolsgF00}
	\hat{N}:=\langle\hat{n}_{\dot{\alpha}}(a)~|~\dot{\alpha}\in\dot{R}^\times,~~a\in U(A)\rangle. 
	\end{equation}
	and
	\begin{equation}\label{eq:MTorElF01}
	\hat{T}:=\langle\hat{h}_{\dot{\alpha}}(a)~|~\dot{\alpha}\in\dot{R}^\times,~~a\in U(A)\rangle\sub \hat{N}. 
	\end{equation} 
	It immediately follows from  (\ref{eq:finiteWF}) that 
	\begin{equation}\label{eq:MTorElF02}
	\hat{n}_{\dot{\alpha}}(a)^{-1}=\hat{n}_{\dot{\alpha}}(-a)
	\end{equation} 
	for $a\in U(A)$, $\dot{\alpha}\in\dot{R}^\times$.
}
\begin{lem}\label{LemStFrExtRelat}
	In the above setting, if $A$ is commutative domain with characteristic zero, then the following relations in $\St_{\dot{R}}(A)$ hold: 
	\begin{description}
		\item[StF3] $\hat{n}_{\dot{\alpha}}{( a)}\hat{x}_{\dot{\beta}}(b)\hat{n}_{\dot{\alpha}}(a)^{-1}=\hat{x}_{w_{\dot\alpha}\dot{\beta}}(c(\dot{\alpha},\dot{\beta})a^{-\langle\dot{\beta},\dot{\alpha}\rangle}b)$,
		\item[StF4] $\hat{n}_{\dot{\alpha}}(a)\hat{n}_{\dot{\beta}}(b)\hat{n}_{\dot{\alpha}}(a)^{-1}=\hat{n}_{w_{\dot\alpha}\dot{\beta}}(c(\dot{\alpha},\dot{\beta})a^{-\langle\dot{\beta},\dot{\alpha}\rangle}b)$,
		\item[StF5]$\hat{n}_{\dot{\alpha}}(a)\hat{h}_{\dot{\beta}}(b)\hat{n}_{\dot{\alpha}}(a)^{-1}=\hat{h}_{w_{\dot\alpha}\dot{\beta}}(c(\dot{\alpha},\dot{\beta})a^{-\langle\dot{\beta},\dot{\alpha}\rangle}b)\hat{h}_{w_{\dot\alpha}\dot{\beta}}(c(\dot{\alpha},\dot{\beta})a^{-\langle\dot{\beta},\dot{\alpha}\rangle})^{-1}$,
		\item[StF6] $\hat{n}_{\dot{\alpha}}(a)=\hat{n}_{-\dot{\alpha}}(c(\dot{\alpha},\dot{\beta})a^{-1})$,
	\end{description}
	where $c(\dot{\alpha},\dot{\beta})=\pm1$ only depends on $\dot{\alpha}$ and $\dot{\beta}$
	and satisfies $c(\dot{\alpha},\dot{\beta})=c(\dot{\alpha},-\dot{\beta})$.\ Hence
	$\hat{T}$ is a normal subgroup of $\hat{N}$. 
\end{lem}
\proof Let $(A)$ denote the field of fractions of $A$.\ The above relations hold
for $\St_{\dot{R}}((A))$ by \cite[Lemma 37]{St67}.\ By the functoriality of $\St_{\dot{R}}(-)$
we have $\St_{\dot{R}}(A)\hookrightarrow\St_{\dot{R}}((A))$.\ Hence the lemma.\qed

\begin{DEF}\label{DefSt}{\rm {Let $R$ be a reduced EARS of type $\dot R$ and $\sg$ a $2$-cocycle on the lattice
			$\Lam=\la R^0\ra$.}\
		We define the Steinberg group $\St_{R,\dot{R},\sigma}({\bbbc})$ associated
		to $R$ over ${\bbbc}$ to be the group generated by symbols $x_{\alpha}(t)$ for each $\alpha\in R^{\times}$ and $t\in {\bbbc}$ subject to the following relations:
		\begin{description}
			\item[St1] $x_{\alpha}(t)$ is additive in $t$.
			\item[St2] {If $\text{rank}(R)\geq2$ consider: for a pair of roots $\alpha=\dot{\alpha}+\delta_{\alpha},\beta=\dot{\beta}+\delta_{\beta}\in R^{\times}\subseteq\dot{R}+\Lambda$, if $\alpha+\beta\in R^{\times}$ (i.e., $\{\alpha,\beta\}$ is a nilpotent pair) then
				\begin{equation}\label{eq:DefSt22}
				(x_{\alpha}(s),x_{\beta}(t))=\prod_{i\alpha+j\beta\in R_{\alpha,\beta}^{+}}x_{i\alpha+j\beta}(c_{ij}{\sg_{i\delta_{\alpha},j\delta_{\beta}}}s^{i}t^{j}),
				\end{equation}
				with the canonical order on the finite set $R_{\alpha,\beta}^{+}$, where $c_{ij}$'s in (\ref{eq:DefSt22}) are chosen as in \cite[Lemma 15]{St67} corresponding to $i\dot{\alpha}+j\dot{\beta}\in \dot{R}$ and
				\begin{equation}\label{eq:DefSt22new} 
				\sg_{i\delta_{\alpha},j\delta_{\beta}}:=\sg(i\d_\a,j\d_\b)\prod_{k=1}^{i-1}\sg(k\d_\a,\d_\a)\prod_{k=1}^{j-1}\sg(\d_\b,k\d_\b).	
				\end{equation}	}	 			
			\item[St$\text{2}'$] If $\text{rank}(R)=1$ then 
			\begin{equation}\label{eq:LWA}
			n_{\alpha}(t)x_{\alpha}(u)n_{\alpha}(t)^{-1}=x_{-\alpha}(-\sigma(\delta_{\alpha},-\delta_{\alpha})^{-1}t^{-2}u),
			\end{equation}
			for $t\in {\bbbc}^{\ast}:=\bbbc\backslash\{0\}$ and $u\in {\bbbc}$, where 
			\begin{equation}\label{eq:WF}
			n_{\alpha}(t):=x_{\alpha}(t)x_{-\alpha}(-\sigma(\delta_{\alpha},-\delta_{\alpha})^{-1}t^{-1})x_{\alpha}(t).
			\end{equation}
			
	\end{description}}
	
\end{DEF}
{
	\begin{cons}\label{ConsSt21}
		In the above setting, if $\alpha+\beta\not\in R^{\times}$ then $(x_{\alpha}(t),x_{\beta}(s))=1$.
	\end{cons}	
	\proof {If $\alpha+\beta\not\in R^{\times}$ then $R_{\alpha,\beta}^{+}=\emptyset$ by Definition~\ref{defrevise}.\ Hence $(x_{\alpha}(t),x_{\beta}(s))=1$ by 
		{\bf St2}. }
	\qed
}	
\begin{rem}\label{RemDStFA}{\rm
		When $R$ is of finite or affine type then the 2-cocycle $\sigma\equiv1$ (see \cite[Lemma 2.10]{AYY15}), hence
		the Steinberg group defined above coincides with the definition of Steinberg groups 
		associated to finite and affine root systems (see \cite[\S 3.6]{Tits87}).}  
\end{rem}

\begin{rem}\label{RemDStWD}{\rm
		When there exists no copy of $G_2$ in $R^{\times}$, then the product in (\ref{eq:DefSt22}) does not depend on the chosen order since in {such cases}
		the right hand side of (\ref{eq:DefSt22}) has either one term ($R_{\alpha,\beta}=A_2$)
		or two terms ($R_{\alpha,\beta}=B_2$).\ And in the latter case, by {Consequence~\ref{ConsSt21}},
		the two terms commute.\ In general, by \cite[Chapter 3 Lemma 18]{St67} 
		the product does not depend on any chosen order for the root set $R^{+}_{\alpha,\beta}$ up to isomorphism.}
\end{rem}

For a reduced extended affine root system $R=R(\rd,S,L)$, we set $\tilde{R}:=R(\rd,\la S\ra,\la L\ra)$.\ {It is shown in \cite[Section 2]{Az99} that $\tilde{R}$ is also an EARS with the same type, rank, nullity and twist number as $R$.}\
We call $\rt$, the $\Lam$-{\it covering} of $R$.

\begin{pro}\label{ProStIso}
	Let $R$ be a reduced EARS of type $\dot{R}$.\ {Then there exists an epimorphism $\chi$ under which $\St_{\dot{R}}(\bbbc_{\sigma})$ is mapped onto $\St_{R,\dot{R},\sigma}(\bbbc)$.\ 
		Furthermore, let $\kappa_{\chi}$ be the kernel of $\chi$.\ Then the following short exact sequence is right-split
		\begin{equation}\label{eq:ProStIso000}
		\kappa_{\chi}\to\St_{\dot{R}}(\bbbc_{\sigma})\stackrel{\chi}{\to}\St_{R,\dot{R},\sigma}(\bbbc).
		\end{equation}
		Hence, 
		\begin{equation}\label{eq:ProStIso0001}
		\St_{\dot{R}}(\bbbc_{\sigma})\cong\kappa_{\chi}\rtimes\St_{R,\dot{R},\sigma}(\bbbc).
		\end{equation}
		In particular, $\St_{\dot{R}}(\bbbc_{\sigma})$ is isomorphic to $\St_{\tilde{R},\dot{R},\sigma}(\bbbc)$.}
\end{pro}

\proof 
Define 
\begin{equation}\label{eq:ProStIso00}
\St_{\dot{R}}(\bbbc_{\sigma})\stackrel{\chi}{\rightarrow}\St_{R,\dot{R},\sigma}(\bbbc),
\end{equation}
where 
\begin{equation}\label{eq:ProStIso01}
\chi(x_{\dot{\alpha}}(\Sigma_{\delta\in\Lambda}t_{\delta}c_{\delta})):=\prod_{\delta\in\Lambda}x_{\dot{\alpha}+\delta}(t_{\delta}).
\end{equation}
Note that (\ref{eq:ProStIso01}) makes sense {since only for finitely many $\d\in\Lam$, the scalars $t_\d$
	%$\Sigma_{\delta\in\Lambda}c_{\delta}t_{\delta}$
	are non-zero}.\ {Note that by convention if $\delta\in\Lambda$ and $\dot{\alpha}+\delta\not\in R^{\times}$ we define $x_{\dot{\alpha}+\delta}(t_{\delta})=1$.}\ We also note that the right hand side of (\ref{eq:ProStIso01}) does not depend on the order of its terms {by Consequence~\ref{ConsSt21}} and that $R$ is reduced.

1) When $\text{rank}(R)\geq2$, by a similar argument as in \cite[Proposition 3.2.11]{Kry95} the assignment $\chi$ defines a group homomorphism.

2) Let $\text{rank}(R)=1$.\ It is enough to show that
\begin{equation}\label{eq:ProStIso03}
\chi(n_{\dot{\alpha}}(tc_{\delta}))\chi(x_{\dot{\alpha}}(rc_{\delta}))\chi(n_{\dot{\alpha}}(-tc_{\delta}))=\chi(x_{-\dot{\alpha}}(\sigma(\delta,-\delta)^{-1}t^{-2}rc_{-\delta})),
\end{equation}
for $\delta\in\Lambda$, $\dot{\alpha}\in\dot{R}$, $t\in \bbbc^{\ast}$ and $r\in \bbbc$.\ But by the definition of $\chi$, for any non-isotropic root $\alpha=\dot{\alpha}+\delta\in R^{\times}$ and any $t\in \bbbc^{\ast}$, the element $\hat{n}_{\alpha}(t)\in\St_{\dot{R},\Lambda,\sigma}(k)$ corresponds to the element $n_{\dot{\alpha}}(c_{\delta}t)\in\St_{\dot{R}}(\bbbc_{\sigma})$.\ Hence the proposition
follows in rank one by the definition of $\chi$ (see (\ref{eq:ProStIso01})) and St$\text{2}'$ in Definition~\ref{DefSt}.\\
In either case,	$\chi$ is onto as for $\a=\dot\a+\d_\a\in R^\times$ and $t\in k^*$, $\chi(x_{\dot\a}(tc_{\d_{\a}}))=x_\a(t)$.\\
Furthermore, in both cases, similar to the way $\chi$ is defined, its right-inverse $\chi_{r}^{-1}$ can be defined by the following assignment:
\begin{equation}\label{eq:ProStIso02}
\chi_{r}^{-1}(x_{\dot{\alpha}+\delta}(t))=x_{\dot{\alpha}}(tc_{\delta}).
\end{equation}
For $\rt$, the map $\chi_{r}^{-1}$ becomes onto, in particular $\chi$ for $\tilde{R}$ is an isomorphism.
\qed

\begin{rem}\label{test000}
	Let $R=R(\rd,S,L)=R(\rd',S',L')$ be two descriptions of $R$ in the form (\ref{eq99-1}). Let $\psi$ be an isomorphism of
	$R$ with $\psi(\rd)=\rd'$. The isomorphism $\psi_{|_\Lam}:\Lam\rightarrow\Lam$ induces a $2$-cocycle
	$\sg':=\psi^{\ast}\sg$.\ Then we have the commutative diagram
	\begin{equation}
	\xymatrix{
		&\St_{\dot{R}}(\bbbc_{\sigma}) \ar@{>}[d]_{\chi}\ar@{>}[r]^{\hat{\psi}}
		&\St_{\dot{R}'}(\bbbc_{\sg'})  \ar@{>}[d]^{\chi'}\\
		&\St_{R,\dot{R},\sigma}(\bbbc) \ar@{>}[r]^{\hat\psi}&\St_{R,\dot{R}',\sigma'}(\bbbc)
	}
	\end{equation}
	where the upper and lower $\hat\psi$'s are the isomorphisms induced by $\psi$, namely
	$x_{\dot{\a}}(c_\d)\mapsto x_{\dot\a'}(c_{\d'})$ and $x_{\a}(t)\mapsto x_{\a'}(t)$ respectively, where
	$\dot\a'=\psi(\dot\a)$, $\d'=\psi(\d)$ and $\a'=\psi(\a)$, for $\dot\a\in\rd$, $\d\in\Lam$ and $\a\in R$.
\end{rem}

We call the Steinberg group associated to $\tilde{R}$, the {\it universal Steinberg group} associated to $R$.

Let $N$ denote the subgroup of $\St_{R,\dot{R},\sigma}(\bbbc)$
generated by all $n_{\alpha}(t)$ where $\alpha\in R^{\times}$ and $t\in \bbbc^{\ast}$. I.e., 
\begin{equation}\label{eq:Polsg00}
N:=\langle n_{\alpha}(t)~|~\alpha\in R^{\times},~~t\in \bbbc^{\ast}\rangle. 
\end{equation}
Define 
\begin{equation}\label{eq:TorEl00}
h_{\alpha}(t):=n_{\alpha}(t)n_{\alpha}(1)^{-1}~~~\text{for}~~t\in \bbbc^{\ast},
\end{equation}
\begin{equation}\label{eq:TorEl01}
T:=\langle h_{\alpha}(t)~|~\alpha\in R^{\times},~~t\in \bbbc^{\ast}\rangle\subset N. 
\end{equation}

\begin{lem}\label{LemTNSGN}
	In the above notations, we have $T\unlhd N$.
\end{lem}
\proof 
{When $\dot{R}$ is of type $A_n$ (for some $n\ge2$), this is \cite[Proposition 3.3.1(i)]{Kry95}.\ Otherwise,} by our standing assumption, $\sg$ is commutative
and hence $\bbbc_{\sigma}$ constitutes a commutative (unital) domain with characteristic zero.\ Now by Lemma~\ref{LemStFrExtRelat}, $\hat{T}$ is normal in $\hat{N}\subset\St_{\dot{R}}(\bbbc_{\sigma})$.\ Under the epimorphism $\chi$ in Proposition~\ref{ProStIso}, $N$ and $T$ are the images 
of $\hat{N}$ and $\hat{T}$ respectively. The proposition follows.\qed

\begin{DEF}\label{DefEAGR}{\rm
		In Definition~\ref{DefSt} if we add the following relations, the resulting 
		group $\G_{R,\dot{R},\sigma}(\bbbc)$ is called the {\it extended affine Kac--Moody group} associated to $(R,\dot{R},\sigma)$:
		\begin{description}
			\item[Tor] For every $\alpha\in R^{\times}$, $h_{\alpha}(t)$ is multiplicative in $t\in \bbbc^{\ast}$.
	\end{description}}
\end{DEF}

\begin{rem}\label{remIsoEAG}{\rm
		Note that since $\chi$ in Proposition~\ref{ProStIso} maps $\hat{h}_{\dot{\alpha}}(tc_{\delta})$ to
		$h_{\dot{\alpha}+\delta}(t)$, it can easily be checked that $\chi$ preserves the multiplicativity relation {\bf Tor} in Definition~\ref{DefEAGR}.\ Let $\tilde{R}$ be the $\Lam$-covering of $R$.\ Then
		\begin{equation}\label{eq:RemIsoEAG00}
		\G_{\dot{R}}(\bbbc_\sigma)\stackrel{\chi}{\cong}\G_{\tilde{R},\dot{R},\sigma}(\bbbc),
		\end{equation}
		where $\G_{\dot{R}}$ is the universal Chevalley group scheme.}	
\end{rem}

\subsection{Groups Associated to EALAs}\label{subsecGAELI}
Here we construct certain groups associated to EALAs by means of certain representations of
Lie algebras; {this method is known as  the ``Kac-Peterson method'' (see for example \cite{Kac85}), to achieve this 
	we follow closely \cite{Kry95}.}\

Let $\ll$ be an EALA with a reduced EARS $R=R(\dot{R},S,L)$.\ Let $G_{\text{nil}}$
($G_{\text{int}}$) be the group associated to $\ll$ with respect to nilpotent (integrable) representations of $\ll$.\ Similarly, let $G_{\text{nil,c}}$ ($G_{\text{int,c}}$) be the group associated to $\ll_{c}$ with respect to nilpotent (integrable) representations of $\ll_{c}$ where $\ll_{c}$
is the core of $\ll$ (see \cite[Section 3.2]{Kry95} {and \cite[Section 6.1]{MP95}} for details).\ To emphesize on the underlined filed we sometimes write $G(\bbbc)$ instead of $G$, for $G\in\{G_{\text{nil}},G_{\text{int}},\ldots\}$.

We generalize \cite[Proposition 3.2.29]{Kry95} for any EALA with a reduced EARS.\ {Recall our standing assumption that in the case $\dot{R}$ is not of type $A_n$ (for some $n\in\bbbn$) we always assume that {the $2$-cocycle} $\sigma$ is commutative.}
\begin{pro}\label{ProisoimageSt}
	Let $\ll$ be an EALA with a reduced EARS $R=R(\dot{R},S,L)$ and the coordinate ring $\bbbc_{\sigma}$.\ 
	Then for each $G\in\{G_{\text{nil}},G_{\text{int}},G_{\text{nil,c}},G_{\text{int,c}}\}$ there exist epimorophisms $\Pi_{G}$ and $\pi_{G}$ such that the following diagram commutes.
	\begin{equation}  \label{eq:ProisoimageSt00}
	\xymatrix{
		&\St_{\dot{R}}(\bbbc_{\sigma}) \ar@{>}[d]_{\chi}\ar@{>}[r]^{\Pi_{G}}
		&G \\
		&\St_{R,\dot{R},\sigma}(\bbbc)\ar@{>}[ru]_{\pi_{G}}}
	\end{equation}
\end{pro}
\proof 
When $\dot{R}=A_n$ (for some $n\geq2$) this is \cite[Proposition 3.2.29]{Kry95}.\ Assume that $\dot{R}$ is not of type $\dot{R}=A_n$ (for any $n\geq2$).

For every $ \dot{\alpha}\in\dot{R}^\times$ and $\delta\in\Lambda$ and any nilpotent representation $\rho$ of $\ll$ define
\begin{equation}\label{eq:ProisoimageSt01}
\Pi^{\rho}_{G_{\text{nil}}}(\hat{x}_{\dot{\alpha}}(tc_{\delta}))=\bar{\rho}(\exp(tc_{\delta}X_{\dot{\alpha}})),
\end{equation}
and similarly for $ \alpha=\dot{\alpha}+\delta_{\alpha}\in R^{\times}$, 
\begin{equation}\label{eq:ProisoimageSt02}
\pi^{\rho}_{G_{\text{nil}}}(x_{\a}(t))=\bar{\rho}(\exp(tc_{\delta_{\alpha}}X_{\dot{\alpha}})),
\end{equation}
where $X_{\dot{\alpha}}$ is chosen as in \cite[Lemma 15]{St67} associated to $\dot{\alpha}\in\dot{R}^\times$
and $\bar{\rho}$ is the induced representation of $G_{\text{nil}}$ corresponding to $\rho$,
and $t\in k$.\ If $\Pi^{\rho}_{G_{nil}}$ and $\pi^{\rho}_{G_{nil}}$ define homomorphisms, then 
clearly by their definitions and the definition of $\chi$
{in Proposition~\ref{ProStIso}}, the diagram (\ref{eq:ProisoimageSt00}) is commutative.\ To 
show that $\Pi^{\rho}_{G_{nil}}$ and $\pi^{\rho}_{G_{nil}}$ are homomorphisms, it suffices to check that the defining relations in $\St_{\dot{R}}(\bbbc_{\sigma})$ {and $\St_{R,\dot{R},\sigma}(k)$} hold for the image of elements in $G_{\text{nil}}$ via $\Pi^{\rho}_{G_{nil}}$ and $\pi^{\rho}_{G_{nil}}$ {respectively}.\ It is clear that for such elements {(St1)} in Definition~\ref{DefSt} holds.\

1) Assume $\text{rank}(\dot{R})\ge2$.
If $\alpha+\beta\in R^{\times}$, then, by Lemma~\ref{lemr2srs}, $R_{\alpha,\beta}$ is an irreducible reduced finite 
root system and $\ll_{\alpha,\beta}$ is a finite-dimensional simple Lie algebra.\ Let $(\bbbc_{\sigma})$ be the fraction field of $\bbbc_{\sigma}$.\ By \cite[Lemma 15]{St67}, {relation (St2) of Definition \ref{DefSt}} holds in the Lie algebra $\ll_{\alpha,\beta}$ over $(\bbbc_{\sigma})$.\ {Hence (St2)} also holds for coefficients from $\bbbc_{\sigma}\subset(\bbbc_{\sigma})$.\ Note that in the case of $\pi^{\rho}_{G}$, this argument is still valid since for every pair of nilpotent roots $\a$ and $\beta$, $R_{\a,b}$ is a finite root system by Lemma~\ref{lemr2srs} therefore all of its non-zero roots are non-isotropic and hence none of them is in the kernel of $\pi^{\rho}_{G}$.\

2) Assume $\text{rank}(\dot{R})=1$.
For $\alpha=\dot{\alpha}+\delta_{\alpha}\in R^{\times}$ by (\ref{eq:ProisoimageSt01}) we have
\begin{equation}\label{eq:ProisoimageSt03}
\Pi^{\rho}_{G_{\text{nil}}}(\hat{x}_{\dot{\alpha}}(tc_{\delta_{\alpha}}))=\bar{\rho}\left(\left(
\begin{array}{cc}
1 & tc_{\delta_{\alpha}} \\
0 & 1 \\
\end{array}
\right)\right),
\end{equation}
for $t\in \bbbc$ and  
\begin{equation}\label{eq:ProisoimageSt04}
\Pi^{\rho}_{G_{\text{nil}}}(\hat{n}_{\dot{\alpha}}(tc_{\delta_{\alpha}}))=\bar{\rho}\left(\left(
\begin{array}{cc}
0 & tc_{\delta_{\alpha}} \\
-\sigma(\delta_{\alpha},-\delta_{\alpha})^{-1}t^{-1}c_{-\delta_{\alpha}} & 0 \\
\end{array}
\right)\right),
\end{equation} 
for $t\in \bbbc^{\ast}$.\ Now, it is a simple matrix calculation to see Definition~\ref{DefSt}(St$\text{2}'$) holds as well. The same argument is valid for $\pi^{\rho}_{G_{\text{nil}}}$. Moreover, elements of the form {$\bar\rho (\exp(tc_{\delta_{\alpha}}X_{\dot{\alpha}}))$} generate $G_{\text{nil}}$.\ Hence 
$\Pi^{\rho}_{G_{\text{nil}}}$ (similarly $\pi^{\rho}_{G_{\text{nil}}}$) induces an
epimorphism.
Since the above argument is independent of the choice of representation and it is
valid also for any $G\in\{G_{\text{nil}},G_{\text{int}},G_{\text{nil,c}},G_{\text{int,c}}\}$, the proposition follows.\qed

For a non-isotropic root $\a\in R^{\times}$, let $G_{\ast}^{\a}(\bbbc)$ be the group in 
$G_{\ast}(\bbbc)$ generated by {$\mathfrak{U}_{\pm\a}(\bbbc):=\langle\Pi_{G_{\ast}}(\hat{x}_{\pm\dot{\alpha}}(tc_{\pm\delta}))~|~t\in\bbbc\rangle$ ($\Pi_{G_{\ast}}$ as in Proposition~\ref{ProisoimageSt})} where 
$$G_{\ast}(\bbbc)\in\{G_{\text{nil}}(\bbbc),G_{\text{int}}(\bbbc),G_{\text{nil,c}}(\bbbc),G_{\text{int,c}}(\bbbc)\}.$$
For such a non-isotropic root $\alpha$ we also know that
$\mathfrak{sl}^{\a}_{2}:= \ll_{\alpha}+[\ll_{\a},\ll_{-\a}]+\ll_{-\alpha}$ is a copy of $\mathfrak{sl}_{2}$ with a standard basis $\{e_{\alpha},h_{\alpha},f_{\alpha}\}$.\

The following is a generalization of \cite[Proposition 3.2.30]{Kry95}.
\begin{pro}\label{ProSL2mod}
	Let $\ll$ be an EALA with a reduced EARS $R$ over $\bbbc$.\ 
	Let $(\rho,M)$ be an integrable $\ll$-module.\ Then the following hold.
	\begin{itemize}
		\item [(i)] For
		every non-isotropic root $\a\in R^{\times}$ there exists an epimorphism 
		\begin{equation}\label{eq:ProSL2modi00}
		\varphi_{\rho}^{\a}:\SL(\bbbc)\to\bar{\rho}(G_{\ast}^{\a}(\bbbc)),
		\end{equation}
		such that
		\begin{equation}\label{eq:ProSL2modi01}
		\varphi_{\rho}^{\a}\left(\left(
		\begin{array}{cc}
		1 & t \\
		0 & 1 \\
		\end{array}
		\right)\right)=\bar{\rho}(\exp te_{\a}),~~~\varphi_{\rho}^{\a}\left(\left(
		\begin{array}{cc}
		1 & 0 \\
		t & 1 \\
		\end{array}
		\right)\right)=\bar{\rho}(\exp tf_{\a}),
		\end{equation}
		for all $t\in \bbbc$, where $(\bar{\rho},M)$ is the representation induced by $(\rho,M)$ of $G_{\ast}(\bbbc)\in\{G_{\text{int}}(\bbbc),G_{\text{int,c}}(\bbbc)\}$. 
		\item [(ii)] Assume either $\text{rank}(R)\geq2$ or the nullity of $\ll$ is 1.\ Then there exists a unique isomorphism 
		\begin{equation}\label{eq:ProSL2modii00}
		\varphi^{\a}:\SL(\bbbc)\to G_{\ast}^{\a}(\bbbc),
		\end{equation}	
		such that for all $t\in \bbbc$
		\begin{equation}\label{eq:ProSL2modii01}
		\varphi^{\a}\left(\left(
		\begin{array}{cc}
		1 & t \\
		0 & 1 \\
		\end{array}
		\right)\right)=\exp te_{\a},~~~\varphi^{\a}\left(\left(
		\begin{array}{cc}
		1 & 0 \\
		t & 1 \\
		\end{array}
		\right)\right)=\exp tf_{\a},
		\end{equation}
		and
		\begin{equation}\label{eq:ProSL2modii03}
		\bar{\rho}\circ\varphi^{\a}=\varphi_{\rho}^{\a},
		\end{equation}
		where $G_{\ast}(\bbbc)\in\{G_{\text{int}}(\bbbc),G_{\text{int,c}}(\bbbc)\}$.
	\end{itemize} 
\end{pro}
\proof (i) Follows from the proof of \cite[Proposition 6.1.7(i)]{MP95}.\

(ii) First assume that $\text{rank}(R)\geq2$.\ In this case, let $\a=\dot{\a}+\delta_{\alpha}\in \dot{R}+R^{0}$.\ Since $\text{rank}(R)\geq2$, there exists $\dot{\b}\in\dot{R}$ such that $\dot{\b}\neq\pm\dot{\a}$
and $\dot{\a}+\dot{\b}\in\dot{R}$.\ Moreover, from the classification of finite root systems one can see that in the non-simply laced cases, $\dot\b$ can be chosen such that
$\dot\a+\dot\b\in\rds$. Therefore, in all cases $\a+\dot\b=\dot\a+\dot\b+\d_\a\in R^\times$ and so $\{\a,\dot{\beta}\}$ forms a nilpotent pair in $R$.
Hence, (ii) follows from Lemma~\ref{lemafsubLi} and \cite[Proposition 6.1.7(ii)]{MP95}.\ When the nullity 
is equal to one, $\ll$ is an affine Kac--Moody Lie algebra.\ Therefore, (ii) directly follows from
\cite[Proposition 6.1.7(ii)]{MP95}.\qed

For a non-isotropic root $\a\in R^{\times}$, let $\bar{h}_{\a}(t):=\pi_{G_{\text{int}}}(h_{\a}(t))$
for $t\in \bbbc^{\ast}$ where $h_{\a}(t))$ is as in {(\ref{eq:TorEl00})} and $\pi_{G_{\text{int}}}$ as in Proposition~\ref{ProisoimageSt}.\ Then we have the  following corollary.
%	immediate from Proposition~\ref{ProSL2mod}.
\begin{cor}\label{Cormulttorel}
	For any non-isotropic root $\a\in R^{\times}$ we have
	\begin{equation}\label{eq:Cormulttorel}
	\bar{h}_{\a}(st)=\bar{h}_{\a}(s)\bar{h}_{\a}(t).
	\end{equation}	
\end{cor}
\proof When either $\text{rank}(R)\geq2$ or the EALA in question has nullity 1, the corollary
follows directly from Proposition~\ref{ProSL2mod}.\ When $\text{rank}(R)=1$ and the nullity of
$\ll$ is greater or equal to 2, then the  corollary follows from (\ref{eq:ProisoimageSt03}) and  
(\ref{eq:ProisoimageSt04}) in Proposition~\ref{ProisoimageSt} (more precisely, adjusted formula	for $\pi_{G_{\text{int}}}$, see (\ref{eq:ProisoimageSt02})).\qed

\begin{cons}\label{ConsEChevepi}
	In the above setting, for $G_{\text{int}}(\bbbc)$ there exists an epimorphism induced from $\pi_{G}$ in Proposition~\ref{ProisoimageSt}, namely:
	\begin{equation}  \label{eq:CorEChevepi00}
	\bar{\pi}_{ G_{\text{int}}}:\G_{R,\dot{R},\sigma}(\bbbc)\to G_{\text{int}}(\bbbc).
	\end{equation}	
\end{cons}
\proof It immediately follows from Corollary~\ref{Cormulttorel} that $\pi_{G}$ preserves {\bf Tor} in Definition~\ref{DefEAGR}, hence the result.\qed 

Similar to $\bar{h}_{\a}(t)$, let $\bar{x}_{\a}(t):=\pi_{G_{\text{int}}}(x_{\a}(t))$
for $t\in \bbbc$ and $\bar{n}_{\a}(t):=\pi_{G_{\text{int}}}(n_{\a}(t))$
for $t\in \bbbc^{\ast}$.\ Also we define $\bar{N}$ and $\bar{T}$ correspondingly.
\begin{cons}\label{Consnormalingr}
	$\bar{T}\unlhd\bar{N}$.
\end{cons}
\proof Follows directly from Lemma~\ref{LemTNSGN} and Proposition~\ref{ProisoimageSt}.\qed

\begin{cor}\label{Corliftweylaction}
	With notations as above we have:
	\begin{itemize}
		\item [(i)] When $\text{rank}(R)\geq2$, then
		\[
		\bar{n}_{\a}(t)\bar{x}_{\lambda}(u)\bar{n}_{\a}(t)^{-1}=\bar{x}_{w_{\a}(\lambda)}{(t^{-\left\langle \lam,\a\right\rangle}u)},
		\]
		where $\lambda=\pm\a$, $t\in \bbbc^{\ast}$ and $u\in \bbbc$.
		\item [(ii)] When $\text{rank}(R)=1$, then
		\[
		\bar{n}_{\a}(t)\bar{x}_{\lambda}(u)\bar{n}_{\a}(t)^{-1}={\bar{x}_{-\lam}}(-\sigma(\delta_{\alpha},-\delta_{\alpha})^{-1}{t^{-\left\langle \lam,\a\right\rangle }u),}
		\]
		where $\a=\dot{\a}+\delta_{\alpha}$, $\lambda=\pm\a$, $t\in \bbbc^{\ast}$ and $u\in \bbbc$.
		\item[(iii)]{ In general, if $\sg$ is commutative and $\a=\dot{\a}+\delta_{\a}$ and $\b=\dot{\b}+\delta_{\b}$ are both in $R^{\times}$, then%\todo{double check the coefficients}
			\[
			\bar{n}_{\a}(t)\bar{x}_{\b}(u)\bar{n}_{\a}(t)^{-1}=\bar{x}_{w_{\a}(\b)}(\lam_{\alpha,\beta}{t^{-\left\langle \b,\a\right\rangle}u)},
			\]
			where 
			$$\lam_{\a,\b}=c({\dot\alpha},{\dot\beta})\cdot\sg({-\left\langle \b,\a\right\rangle}\d_\a,\d_\b)^{c({\dot\alpha},\dot{\beta})}\prod_{k=|c({\dot\alpha},{\dot\beta})|}^{|{-\left\langle \b,\a\right\rangle}-c({\dot\alpha},{\dot\beta})|}\sg(k\d_\a,\d_\a)^{c({\dot\alpha},{\dot\beta})}$$
			and $c({\dot\alpha},{\dot\beta})=\pm1$ only depends on ${\dot\alpha}$ and ${\dot\beta}$.}
	\end{itemize}
\end{cor}

\proof (i) and (ii): Similar to the proof of Corollary~\ref{Cormulttorel}.\

(iii): First note that by {StF3} in Lemma~\ref{LemStFrExtRelat}, we have
\[
\hat{n}_{\dot{\alpha}}{( a)}\hat{x}_{\dot{\beta}}(b)\hat{n}_{\dot{\alpha}}(a)^{-1}=\hat{x}_{w_{\dot\alpha}\dot{\beta}}(c(\dot{\alpha},\dot{\beta})a^{-\langle\dot{\beta},\dot{\alpha}\rangle}b),
\]
hence
\[
\begin{array}{l}
\hat{n}_{\dot{\alpha}}{( tc_{\d_{\a}})}\hat{x}_{\dot{\beta}}(uc_{\d_\b})\hat{n}_{\dot{\alpha}}(t\d_{\a})^{-1}=\\
\hat{x}_{w_{\dot\alpha}\dot{\beta}}\big(c(\dot{\alpha},\dot{\beta})t^{{-\left\langle \b,\a\right\rangle}}u\sg({-\left\langle \b,\a\right\rangle}\d_\a,\d_\b)^{c(\dot{\alpha},\dot{\beta})}\prod_k\sg(k\d_\a,\d_\a)^{c(\dot{\alpha},\dot{\beta})}c_{{-\left\langle \b,\a\right\rangle}\d_{\a}+\d_\b}\big),
\end{array}
\] 
where $k$ ranges from $|c(\dot{\alpha},\dot{\beta})|$ to $|-{\left\langle \b,\a\right\rangle}-c(\dot{\alpha},\dot{\beta})|$.
But this {equation is mapped} to the following via $\Pi^{\rho}_{G}$ as (\ref{eq:ProisoimageSt01}) in Proposition~\ref{ProisoimageSt}:
{\[
\bar{n}_{\a}(t)\bar{x}_{\b}(u)\bar{n}_{\a}(t)^{-1}={\bar{x}_{w_{\a}(\b)}(\lam_{{\alpha},{\beta}}t^{-\left\langle \a,\b\right\rangle}u)}.
\] }
\qed

\begin{rem}{\rm
		Note that in Corollary~\ref{Corliftweylaction} the difference between (i) and (ii) is due to
		the fact that in (i) any non-isotropic root can be embedded into an affine root system
		of rank two (see Proposition~\ref{affine-99}) and on the corresponding  affine Lie subalgebra we have $\sigma\equiv1$ by \cite[Lemma 2.10]{AYY15}.\ There is not such embedding when $\text{rank}(R)=1$.}
\end{rem}

\begin{lem}\label{Lemintmodcorext}
	Let $\ll$ be an EALA over $\bbbc$.\ Then every integrable (respectively, nilpotent) $\ll_c$-module can be extended to an integrable (respectively, nilpotent) $\ll$-module.
\end{lem} 
\proof It immediately follows from \cite[Lemma 3.2.32]{Kry95} and the explanation after its proof on Page 130.\qed

Let $K_{\text{nil}}$ (respectively, $K_{\text{int}}$) be the intersection of the kernel of all
nilpotent (respectively, integrable) representations of an EALA $\ll$ over $\bbbc$.\ Similarly, we define 
$K_{\text{nil,c}}$ and $K_{\text{int,c}}$ for the core $\ll_c$ of $\ll$.
\begin{pro}\label{Prokernnilintcon}
	Let $\ll$ be an EALA with a reduced EARS $R$.\ Then the following hold.
	\begin{itemize}
		\item [(i)]  $K_{\text{nil}}\subseteq K_{\text{nil,c}}$.
		\item [(ii)] $K_{\text{int}}\subseteq K_{\text{int,c}}$.
	\end{itemize}
	Therefore, there exist natural epimorphisms as follows
	\begin{equation}\label{eq:Prokernnilintcon00}
	\Psi_{\text{nil}}:G_{\text{nil}}\to G_{\text{nil,c}}, ~~\Psi_{\text{int}}:G_{\text{int}}\to G_{\text{int,c}}.
	\end{equation}
\end{pro}
\proof Follows from Lemma~\ref{Lemintmodcorext} and the proof of \cite[Proposition 3.2.38]{Kry95}.\qed

Define epimorphisms
\begin{equation}\label{eq:epicornil}
\Psi_{\text{nil},c}:=\Psi_{\text{nil}}\circ\pi_{G_{\text{nil}}}:\St_{R,\dot{R},\sigma}(\bbbc)\to G_{\text{nil,c}}(\bbbc), 
\end{equation}
and
\begin{equation}\label{eq:epicorint}
\Psi_{\text{int},c}:=\Psi_{\text{int}}\circ\pi_{G_{\text{int}}}:\St_{R,\dot{R},\sigma}(\bbbc)\to G_{\text{int,c}}(\bbbc),
\end{equation}
where $\pi_{G_{\text{nil}}}$ and $\pi_{G_{\text{int}}}$ are as in Proposition~\ref{ProisoimageSt}.

\begin{cons}\label{ConsepiUniv}
	Let $\ll$ be an EALA with a reduced EARS $R=R(\dot{R},S,L)$ and let $\bbbc_{\sigma}$ be a coordinate ring.\ 
	Then there exists an epimorphism from $\G_{R,\dot{R},\sigma}(\bbbc)$ onto  $G_{\text{int,c}}$, namely
	\begin{equation}  \label{eq:ConsepiUniv00}
	\bar{\pi}_{ G_{\text{int},c}}:\G_{R,\dot{R},\sigma}(\bbbc)\to G_{\text{int},c}(\bbbc).
	\end{equation}	
	
\end{cons}
\proof Follows from Consequence~\ref{ConsEChevepi} and Proposition~\ref{Prokernnilintcon}.\qed

By the above adjustment, the next proposition is a generalization of \cite[Proposition 3.2.41]{Kry95}.\ Its proof is also similar to the proof
of \cite[Proposition 3.2.41]{Kry95}.		

\begin{pro}\label{Proisonilintc} 
	In the above setting we have
	\[
	G_{\text{nil}}(\bbbc)\cong G_{\text{nil},c}(\bbbc)\cong G_{\text{int}}(\bbbc)\cong G_{\text{int,c}}(\bbbc).
	\]
\end{pro}
In the light of Proposition~\ref{Proisonilintc}, from now on we denote any of
the above isomorphic groups simply by $\bar{G}$.\ Also we denote the group corresponding 
to the adjoint representation by $\text{Ad}(G)$.\ Similarly, $\text{Ad}(N)$ and  $\text{Ad}(T)$ denote the groups corresponding to $\bar{N}$ and $\bar{T}$ respectively.\
It is evident from Consequence~\ref{Consnormalingr} that $\text{Ad}(T)\unlhd\text{Ad}(N)$.\ Moreover, elements $\theta_{\alpha}(t):=\text{Ad}(\bar{n}_{\alpha}(t))$ for $\alpha\in R^{\times}$ and
$t\in\bbbc$ are of the following form:
\begin{equation}\label{eq:dwe}
\theta_{\alpha}(t):=\exp\text{ad}te_{\alpha}\exp\text{ad}-t^{-1}f_{\alpha}\exp\text{ad}te_{\alpha},
\end{equation}
where $\mathfrak{sl}^{\a}_{2}:= \ll_{\alpha}+[\ll_{\a},\ll_{-\a}]+\ll_{-\alpha}$ is a copy of $\mathfrak{sl}_{2}$ with a standard basis $\{e_{\alpha},h_{\alpha},f_{\alpha}\}$.\
For a non-isotropic root $\alpha\in R^{\times}$ and an isotropic root $\sg=\sum_{i=1}^{\nu}m_{i}\sg_{i}\in R^{0}$ ($\text{null}(R)=\nu$) let $c_{(\a,\sg)}$ be as in (\ref{b4}).
%\begin{equation}\label{eq:centelem00}
%c_{(\alpha,\delta)}:=(w_{\alpha}w_{\alpha+\delta})(w_{\alpha+\delta_{1}}w_{\alpha})^{m_{1}}\cdots(w_{\alpha+\delta_{\nu}}w_{\alpha})^{m_{\nu}}.
%\end{equation}
\begin{thm}\label{TheWeyliso}
	Let $\ll$ be an EALA over $\bbbc$ with a reduced EARS $R=R(\dot{R},S,L)$.\ Let $\w$ be the Weyl group associated to $\ll$.\ Then 
	\begin{equation}
	\w\stackrel{\Phi}{\cong}\text{Ad}(N)/\text{Ad}(T),
	\end{equation}
	where $\Phi$ sends $w_{\alpha}$ to  $\text{Ad}(T)\text{Ad}(\bar{n}_{\alpha}(t))$ for all
	$\alpha\in R^{\times}$ and $t\in\bbbc^{\ast}$.
\end{thm}
\proof When $\dot{R}=A_{n}$ (for $n\geq2$) this is \cite[Proposition 3.3.1]{Kry95}.\ We now consider all other {possible cases.}\
Similar to the proof of \cite[Proposition 3.3.1]{Kry95}, it is clear
that $\Phi$ does not depend on the choice of $t\in\bbbc^{\ast}$.\
First we show that $\Phi$ defines an epimorphism.\ For this, note that by {Theorem \ref{azam2000}}, in view of Proposition~\ref{ProisoimageSt}, it is enough to show that 
elements of the form $\tilde{w}_{\alpha}:=\text{Ad}(T)\text{Ad}(\bar{n}_{\alpha}(t))$ satisfy	the generalized presentation by conjugation relations for $\w$.\ Recall from {Proposition~\ref{ProisoimageSt}} that elements $\hat{n}_{\dot{\alpha}}(c_{\delta_{\alpha}})$ in $\St_{\dot{R}}(\bbbc_{\sigma})$ correspond to $\text{Ad}(\bar{n}_{\alpha}(1))$ where $\alpha=\dot{\alpha}+\delta_{\alpha}\in R^{\times}$.\ Since both
$\text{Ad}(T)\text{Ad}(\bar{n}_{\alpha}(1))$ and $\text{Ad}(T)\text{Ad}(\bar{n}_{\alpha}(-1))$ are in the class $\tilde{w}_{\alpha}$
by (\ref{eq:MTorElF02}), it is clear that $\tilde{w}_{\alpha}^{2}=1$.\ Hence relation
(i) in {Theorem \ref{azam2000} holds.}\

Let $\text{Rank}(R)\geq2$.\ In this case, for any $\beta=\dot{\beta}+\delta_{\beta}$, by {Lemma~\ref{LemStFrExtRelat}({StF4})}, we have
\begin{equation}\label{eq:TheWeyliso00}
\hat{n}_{\dot{\alpha}}(c_{\delta_{\alpha}})\hat{n}_{\dot{\beta}}(c_{\delta_{\beta}})\hat{n}_{\dot{\alpha}}(c_{\delta_{\alpha}})^{-1}=\hat{n}_{w_{\dot{\alpha}}\dot{\beta}}(\kappa(\alpha,\beta)c_{\delta_{\alpha}}^{-(\dot{\beta},\dot{\alpha})}c_{\delta_{\beta}})\stackrel{\chi}{=}n_{w_{\alpha}\beta}(\kappa(\alpha,\beta)),
\end{equation}
for some $\kappa(\alpha,\beta)\in\bbbc^{\ast}$.\ Therefore, $\tilde{w}_{\alpha}\tilde{w}_{\beta}\tilde{w}_{\alpha}=\tilde{w}_{w_{\alpha}(\beta)}$ and
hence relation (ii) in {Theorem \ref{azam2000} also holds.}\ Moreover, when $\text{Rank}(R)=1$ by (\ref{eq:ProisoimageSt03}) and (\ref{eq:ProisoimageSt04})
and a straightforward matrix calculation one can easily see that relation (ii) also holds in this case.\

Now consider $c_{(\alpha,\delta)}$ as in (\ref{b4}), we have:
\begin{equation}\label{eq:TheWeyliso01}
\tilde{c}_{(\alpha,\delta)}:=(\hat{n}_{\dot{\alpha}}(c_{\delta_{\alpha}+\delta})\hat{n}_{\dot{\alpha}}(c_{\delta_{\alpha}}))(\hat{n}_{\dot{\alpha}}(c_{\delta_{\alpha}})\hat{n}_{\dot{\alpha}}(c_{\delta_{\alpha}+\delta_{1}}))^{m_{1}}\cdots(\hat{n}_{\dot{\alpha}}(c_{\delta_{\alpha}})\hat{n}_{\dot{\alpha}}(c_{\delta_{\alpha}+\delta_{\nu}}))^{m_{\nu}},
\end{equation}
which is equal to 
\begin{equation}\label{eq:TheWeyliso02}
(\hat{h}_{\dot{\alpha}}(c_{\delta_{\alpha}+\delta})\hat{h}_{\dot{\alpha}}(-c_{\delta})^{-1})(\hat{h}_{\dot{\alpha}}(c_{\delta_{1}})\hat{h}_{\dot{\alpha}}(-c_{\delta_{\alpha}+\delta})^{-1})^{m_{1}}\cdots(\hat{h}_{\dot{\alpha}}(c_{\delta_{\nu}})\hat{h}_{\dot{\alpha}}(-c_{\delta_{\alpha}+\delta_{\nu}})^{-1})^{m_{\nu}},
\end{equation}
whose image under $\Pi_{G}\circ\text{Ad}$ ($\Pi_{G}$ as in Proposition~\ref{ProisoimageSt}) is in $\text{Ad}(T)$, so  is the corresponding
element to $\prod_{p=1}^{n}\tilde{c}^{\epsilon_p}_{(\alpha_p,\delta_p)}$ for any
collection $\{(\epsilon_p,\alpha_p,\delta_p)\}^{n}_{p=1}$.\ Hence
relation (iii) in {Theorem \ref{azam2000}} holds as well.\

At last, by \cite[(1.26)]{AABGP97} we know that the restriction
of the action of $\theta_{\alpha}(t)$ (for all $\alpha\in R^{\times}$ and $t\in\bbbc^{\ast}$)
on the Cartan subalgebra $\hh$ of $\ll$ is the same as the action of the Weyl group
element $w_{\alpha}$, so the theorem follows.\qed

\noindent{\bf Acknowledgements.}
The second-named author would like to thank Ralf K{\"o}hl for many useful discussions on the preliminary version of this article during a visit to Gie\ss en university.
%\end{acknowledgements}


\bibliographystyle{siam}     
%\bibliography{Ref}

\end{document}